\documentclass{article}

\usepackage{hyperref}
\usepackage[dvipsnames]{xcolor}

\usepackage{todonotes}

\usepackage{graphicx}%
\usepackage{multirow}%
\usepackage{amsmath,amssymb,amsfonts}%
\usepackage{amsthm}%
\usepackage{mathrsfs}%
\usepackage[figuresright]{rotating}%
\usepackage[title]{appendix}%
\usepackage{xcolor}%
\usepackage{textcomp}%
\usepackage{manyfoot}%
\usepackage{booktabs}%

\makeatletter

\DeclareOldFontCommand{\rm}{\normalfont\rmfamily}{\mathrm}%
\DeclareOldFontCommand{\sf}{\normalfont\sffamily}{\mathsf}%
\DeclareOldFontCommand{\tt}{\normalfont\ttfamily}{\mathtt}%
\DeclareOldFontCommand{\bf}{\normalfont\bfseries}{\mathbf}%
\DeclareOldFontCommand{\it}{\normalfont\itshape}{\mathit}%
\DeclareOldFontCommand{\sl}{\normalfont\slshape}{\@nomath\sl}%
\DeclareOldFontCommand{\sc}{\normalfont\scshape}{\@nomath\sc}%
\DeclareRobustCommand*\cal{\@fontswitch\relax\mathcal}%
\DeclareRobustCommand*\mit{\@fontswitch\relax\mathnormal}%

\renewcommand\normalsize{%
   \@setfontsize\normalsize{10bp}{12bp}%
   \abovedisplayskip 12\p@ \@plus2\p@ \@minus1\p@
   \abovedisplayshortskip \z@ \@plus3\p@%
   \belowdisplayshortskip 3\p@ \@plus3\p@ \@minus3\p@%
   \belowdisplayskip \abovedisplayskip%
   \let\@listi\@listI}%
\normalsize%

\newcommand\medsize{%
   \@setfontsize\small\@xipt{13}%
   \abovedisplayskip 5\p@ \@plus3\p@ \@minus4\p@
   \abovedisplayshortskip \z@ \@plus2\p@
   \belowdisplayshortskip 3\p@ \@plus2\p@ \@minus2\p@
   \def\@listi{\leftmargin\leftmargini
               \topsep 4\p@ \@plus2\p@ \@minus2\p@
               \parsep 2\p@ \@plus\p@ \@minus\p@
               \itemsep \parsep}%
   \belowdisplayskip \abovedisplayskip}%
\renewcommand\small{%
   \@setfontsize\small\@ixpt{11}%
   \abovedisplayskip 5\p@ \@plus3\p@ \@minus4\p@
   \abovedisplayshortskip \z@ \@plus2\p@
   \belowdisplayshortskip 3\p@ \@plus2\p@ \@minus2\p@
   \def\@listi{\leftmargin\leftmargini
               \topsep 4\p@ \@plus2\p@ \@minus2\p@
               \parsep 2\p@ \@plus\p@ \@minus\p@
               \itemsep \parsep}%
   \belowdisplayskip \abovedisplayskip}%
\renewcommand\footnotesize{%
   \@setfontsize\footnotesize{7}{8}%
   \abovedisplayskip 5\p@ \@plus2\p@ \@minus4\p@
   \abovedisplayshortskip \z@ \@plus\p@
   \belowdisplayshortskip 3\p@ \@plus\p@ \@minus2\p@
   \def\@listi{\leftmargin\leftmargini
               \topsep 3\p@ \@plus\p@ \@minus\p@
               \parsep 2\p@ \@plus\p@ \@minus\p@
               \itemsep \parsep}%
   \belowdisplayskip \abovedisplayskip}
\renewcommand\scriptsize{\@setfontsize\scriptsize\@ixpt\@ixpt}%
\newcommand\scrisize{\@setfontsize\scrisize{9.3}{9}}%
\renewcommand\tiny{\@setfontsize\tiny\@vpt\@vipt}%
\renewcommand\large{\@setfontsize\large{12}{14}}%
\newcommand\larg{\@setfontsize\larg{11}{13}}%
\renewcommand\Large{\@setfontsize\Large{16}{18}}%
\renewcommand\LARGE{\@setfontsize\LARGE\@xviipt{22}}%
\renewcommand\huge{\@setfontsize\huge\@xxpt{25}}%
\renewcommand\Huge{\@setfontsize\Huge\@xxvpt{30}}%
\DeclareMathSizes{\@ixpt}{\@ixpt}{7}{5}%
\DeclareMathSizes{\@xpt}{\@xpt}{7}{5}%
\DeclareMathSizes{\@xipt}{\@xipt}{7}{5}%
\DeclareRobustCommand*\textsubscript[1]{%
  \@textsubscript{\selectfont#1}}%
\def\@textsubscript#1{%
  {\m@th\ensuremath{_{\mbox{\fontsize\sf@size\z@#1}}}}}%

\usepackage[%
    reversemp,
    paperwidth=210mm,
    paperheight=297mm,
    top={26mm},
    headheight={5.5pt},
    headsep={5.6mm},
    text={31pc,194.25mm},
    marginparsep=5mm,
    marginparwidth=12mm,
    bindingoffset=6mm,
    footskip=10mm]{geometry}
\def\eqnarray{%
   \stepcounter{equation}%
   \def\@currentlabel{\p@equation\theequation}%
   \global\@eqnswtrue
   \m@th
   \global\@eqcnt\z@
   \tabskip\@centering
   \let\\\@eqncr
   $$\everycr{}\halign to\displaywidth\bgroup
       \hskip\@centering$\displaystyle\tabskip\z@skip{##}$\@eqnsel
      &\global\@eqcnt\@ne\hskip \tw@\arraycolsep \hfil${##}$\hfil
      &\global\@eqcnt\tw@ \hskip \tw@\arraycolsep
         $\displaystyle{##}$\hfil\tabskip\@centering
      &\global\@eqcnt\thr@@ \hb@xt@\z@\bgroup\hss##\egroup
         \tabskip\z@skip
      \cr
}
\def\endeqnarray{%
      \@@eqncr
      \egroup
      \global\advance\c@equation\m@ne
   $$\@ignoretrue
}
\def\raggedright{\rightskip0pt plus 1fil\parfillskip=0pt\relax}%
\def\raggedcenter{\leftskip=0pt plus 0.5fil\rightskip=0pt plus 0.5fil%
\parfillskip=0pt\let\hb=\break}%
\def\titraggedcenter{\leftskip=12pt plus 0.5fil\rightskip=12pt plus 0.5fil%
\parfillskip=0pt\let\hb=\break}%
\def\absraggedcenter{\leftskip=24pt plus 0.5fil\rightskip=24pt plus 0.5fil%
\parfillskip=0pt\let\hb=\break}%

\makeatother

\makeatletter

\def\@begintheorem#1#2[#3]{%
  \deferred@thm@head{\the\thm@headfont \thm@indent
    \@ifempty{#1}{\let\thmname\@gobble}{\let\thmname\@iden}%
    \@ifempty{#2}{\let\thmnumber\@gobble}{\let\thmnumber\@iden}%
    \@ifempty{#3}{\let\thmnote\@gobble}{\let\thmnote\@iden}%
    \thm@swap\swappedhead\thmhead{#1}{#2}{#3}%
    \the\thm@headpunct
    \thmheadnl 
    \hskip\thm@headsep
  }%
  \ignorespaces
}

\def\@endtheorem{\endtrivlist\@endpefalse}

\makeatother

\hypersetup{%
        colorlinks,
        breaklinks=true,
        plainpages=false,%
        citecolor=blue,
        linkcolor=blue,
        urlcolor=blue,
        bookmarksopen=true,%
        bookmarksnumbered=false,%
        bookmarksdepth=5%
}

\usepackage[T1]{fontenc}
\usepackage{lmodern}
\usepackage{microtype}

\usepackage[en-us]{datetime2}

\newtheorem{theorem}{Theorem}

\newtheorem{lemma}[theorem]{Lemma}
\newtheorem{proposition}[theorem]{Proposition}
\newtheorem{corollary}[theorem]{Corollary}
\theoremstyle{definition}

\theoremstyle{remark}
\newtheorem{remark}[theorem]{Remark}

\newcommand{\D}{\mathrm{d}}

\begin{document}

\title{A note on the NLS equation on Cartan-Hadamard manifolds with unbounded and vanishing potentials}
\author{Luigi Appolloni \\ \small Dipartimento di Matematica e Applicazioni \\ \small Università degli Studi di Milano Bicocca \\ \small via Roberto Cozzi 55, I-20125 Milano \and Giovanni Molica Bisci \\ \small Dipartimento di Scienze Pure e Applicate \\ \small Università di Urbino Carlo Bo \\ \small piazza della Repubblica 13,  I-61029 Urbino \and Simone Secchi \\ \small Dipartimento di Matematica e Applicazioni \\ \small Università degli Studi di Milano Bicocca \\ \small via Roberto Cozzi 55, I-20125 Milano}



\date{\DTMnow}
\maketitle

\begin{abstract}
  We study the semilinear equation \(-\Delta_g u + V(\sigma) u = f(u)\) on
  a Cartan-Hadamard manifold \(\mathcal{M}\) of dimension \(N \geq 3\), and we prove
  the existence of a nontrivial solution under suitable assumptions on
  the potential function \(V \in C(\mathcal{M})\). In particular, the
  decay of \(V\) at infinity is allowed, with some restrictions
	related to the geometry of \(\mathcal{M}\). We generalize some results proved in $\mathbb{R}^N$ by Alves \emph{et al.}, see \cite{Alves_2012}.
\end{abstract}

\section{Introduction}

The study of Nonlinear Schrödinger Equations of the form
\begin{equation} \label{eq:NLS} -\Delta u + V(x) u = f(u)
  \quad\text{in $\mathbb{R}^N$}
\end{equation}
is a classical research topic in the theory of Partial Differential
Equations, and the lack of compactness introduced by the non-compact
space $\mathbb{R}^N$ is an obstruction to the application of basic
tools of Nonlinear Analysis. For example, Variational Methods
typically require the validity of some \emph{compactness condition}
for the so-called Palais-Smale sequences associated to the previous
equation. At this point, suitable conditions on the nonlinearity $f$
and on the potential function $V \colon \mathbb{R}^N \to \mathbb{R}$
must be imposed in order to find solutions.

The basic case in which $V$ coincides with a positive constant was
studied in \cite{Berestycki_1983}, while P.H. Rabinowitz considered in
\cite{Rabinowitz_1992} the case of coercive potentials. i.e.
\begin{equation*}
	0 < \inf_{x \in \mathbb{R}^N} V(x) < \liminf_{\vert x \vert \to +\infty} V(x).
\end{equation*}
While standard embedding theorems for weighted Sobolev spaces ensure the existence of solutions under the strong assumption
\begin{equation*}
	\lim_{\vert x \vert \to +\infty} V(x) = +\infty,
\end{equation*}
a milder condition was introduced in \cite{Bartsch_1995}: for every $M>0$, the Lebesgue measure of the set
\begin{equation*}
	\left\lbrace x \in \mathbb{R}^N \mid V(x) \leq M \right\rbrace
\end{equation*}
must be finite. In all these cases, the potential $V$ may be unbounded
from above, but it has to be bounded away from zero on the whole
$\mathbb{R}^N$. Mathematically, this implies that equation
\eqref{eq:NLS} can be successfully set in the Sobolev space
$H^1(\mathbb{R}^N)$.

On the other hand, the presence of a vanishing potential $V$, in the sense that $\inf_{x \in \mathbb{R}^N} V(x)=0$, introduces additional difficulties, starting from the fact that solutions need not lie in $L^2(\mathbb{R}^N)$. We refer to \cite{zbMATH05720849,Ambrosetti_2005,Benci_2005} for some recent existence results, which have been extended in several directions later. The interested readers can consult also \cite{MR2254489},  \cite{MR4232664}, \cite{MR2002047}, \cite{MR1899450} and the reference therein for a complete summary on the already existing results. We point out that in all the cited papers the NLS equation \eqref{eq:NLS} is set in the standard Euclidean space $\mathbb{R}^N$.

\medskip

Although the Schr\"odinger equation in $\mathbb{R}^N$ has been
extensively studied, there is a surprising lack of understanding when
it comes to looking for solutions for the equation on non-Euclidean
spaces such as Riemannian Manifolds. One of the first contributions in
this direction is given in the papers \cite{MR4020749} and
\cite{MR3490853} where the authors proved the existence of solutions
for the Schr\"odinger equation or for the Schr\"odinger-Maxwell system
requiring suitable bounds on the Ricci or sectional curvature. More
recently, the same authors of this paper proved respectively in
\cite{MR4540772} and \cite{MR4511374} the existence of three solutions
for the Schr\"odinger equation on a manifold with asymptotically
non-negative Ricci curvature with a coercive potential and the
existence of infinitely many solutions on a Cartan-Hadamard manifold
with a constant potential and an oscillatory nonlinearity. We also
refer to \cite[Part III, Chapter 8]{molica2021nonlinear} for recent
results about nonlinear equations on Cartan-Hadamard manifolds.

In this note we deal with a semilinear elliptic problem of the form
\begin{equation}
\left\lbrace
\begin{array}{ll}
-\Delta_g u + V(\sigma) u = f(u) &\hbox{on $\mathcal{M}$}\\
u>0 &\hbox{on $\mathcal{M}$},
\end{array}
\right.
\label{eq:P}
\end{equation}
where $\mathcal{M}$ is a \emph{non-compact} manifold of dimension
$N \geq 3$, $V$ is a real-valued continuous potential function on
$\mathcal{M}$, and $f \colon \mathbb{R} \to \mathbb{R}$ is a
continuous function. Inspired by the paper~\cite{Alves_2012}, we prove
the existence of a positive solution under suitable assumptions, for
which we refer to Section \ref{sec:2}. To the best of our knowledge,
our result is new in the framework of Cartan-Hadamard manifolds.

\medskip

It should be noted that the Riemannian setting may differ considerably
from the Euclidean one for several reasons. First of all, the behavior
of the potential $V$ \emph{at infinity} requires a good replacement
for the basic condition $\Vert x \Vert \to +\infty$ in
$\mathbb{R}^N$. The Riemannian distance from a \emph{fixed} point is
the first attempt, but of course the condition at infinity must be
independent of the choice of local charts. On the other hand, the
topological approach which consists in saying that a sequence diverges
to infinity if and only if it escapes every compact subset, may be too
weak for a quantitative analysis.

A good compromise is the use of \emph{manifolds with a pole}, see \cite{kasue1981riemannian} and the references therein. By definition, a Riemannian manifold  has a pole $o$ if and only if the exponential map at $o$ induces a global diffeomorphism. This allows us to replace the Euclidean norm of $\mathbb{R}^N$ with the distance from the pole $o$, $d_g(\cdot,o)$.

\smallskip

But even under very specific assumptions on the Riemaniann metric, the geometry of the manifold puts in jeopardy some standard tricks of Nonlinear Analysis, like the use of cut-off functions. See Remark \ref{rem:6} below.

\section{Quick survey of model manifolds and main result} \label{sec:2}

We will work on a Cartan-Hadamard manifold $\mathcal{M}$ of dimension
$N \geq 3$, i.e.  a simply-connected complete non-compact manifold
with non-positive sectional curvature. As is well know, the cut locus
of any point of $\mathcal{M}$ is empty, hence $\mathcal{M}$ is a
manifold with a pole. We collect some basic information about
Riemannian geometry in our setting. We follow
\cite{monticelli2020nonexistence}, and we refer to
\cite{Grigor_yan_1999} for more details.

\begin{itemize}
  \item We fix a pole $o \in \mathcal{M}$, which we
  will consider as the \emph{origin}.  For every
  $\sigma \in \mathcal{M}$, $\sigma \neq o$, we may define polar
  coordinates as follows: we let $r = d_g(\sigma,o) >0$ and $\theta$
  be an angle such that the shortest geodesic from $o$ to $\sigma$
  starts with direction $\theta$ in the tangent space
  $T_\sigma \mathcal{M}$. Since $T_\sigma \mathcal{M}$ can be
  identified with $\mathbb{R}^N$, the angle $\theta$ may be seen as an
  element of the sphere $\mathbb{S}^{N-1}$.

\item The Riemannian metric of $\mathcal{M}$ is expressed in polar coordinates as
\begin{equation*}
	g = \D r^2 + A_{ij}(r,\theta) \D \theta^i \, \D \theta^j
\end{equation*}
for some positive-definite matrix $[A_{ij}]$. Here $(\theta^1,\ldots,\theta^{N-1})$ are local coordinates on $\mathbb{S}^{N-1}$.

\item The Laplace-Beltrami operator is then written in the form
\begin{equation*}
	\Delta_g = \frac{\partial^2}{\partial r^2} +
	\mathcal{F}(r,\theta) \frac{\partial}{\partial r} + \Delta_{S_r},
\end{equation*}
where
\begin{equation*}
	\mathcal{F}(r,\theta) = \frac{\partial}{\partial r} \log \sqrt{\det A_{ij}(r,\theta)},
\end{equation*}
and $\Delta_{S_r}$ is the Laplace-Beltrami operator on the sphere $S_r = \partial B_r(o)$.

\item	A model manifold is a manifold with a pole such that the Riemannian metric has the form
	\begin{equation*}
		g = \D r^2 + h(r)^2 \, \D \theta^2,
	\end{equation*}
	where $\D \theta^2=\beta_{ij}\, \D \theta^i \, \D \theta^j$ is the standard metric on $\mathbb{S}^{N-1}$ and $h$ is a  smooth function such that $h(0)=0$, $h'(0)=1$ and $h(r)>0$ for $r>0$.
%
\item On a model manifold the Laplace-Beltrami operator is expressed as
\begin{equation}
\label{eq:laplace-beltrami}
	\Delta_g = \frac{\partial^2}{\partial r^2} + (N-1) \frac{h'(r)}{h(r)} \frac{\partial}{\partial r} + \frac{1}{h(r)^2} \Delta_{\mathbb{S}^{N-1}}.
\end{equation}


\item It follows from \eqref{eq:laplace-beltrami} that harmonic functions on a model manifold $\mathcal{M}$ endowed with the metric $g=dr^2 + h(r) \D \theta^2$ must satisfy the equation
\begin{equation*}
	\frac{\partial^2 w}{\partial r^2} + (N-1) \frac{h'(r)}{h(r)} \frac{\partial w}{\partial r} + \frac{1}{h(r)^2} \Delta_{\mathbb{S}^{N-1}} w =0.
\end{equation*}
If we look for \emph{radially symmetric} harmonic functions $w=w(r)$, i.e. harmonic functions depending only on the variable $r=d_g(\sigma,o)$, the previous equation reduces to
\begin{equation*}
	\frac{\partial^2 w}{\partial r^2} + (N-1) \frac{h'(r)}{h(r)} \frac{\partial w}{\partial r} =0.
\end{equation*}
Under mild assumptions on the function $h$, this equation can be solved to find the two-parameter family of harmonic functions
\begin{equation*}
	w(r) = c_1 + c_2 \int_1^r \frac{\D t}{h(t)^{N-1}},
\end{equation*}
for any $c_1 \in \mathbb{R}$, $c_2 \in \mathbb{R}$.
We will assume that
\begin{itemize}
	\item[$(h)$] $h$ is a positive smooth function such that the improper integral
\begin{equation*}
	\int_1^{+\infty}  \frac{\D t}{h(t)^{N-1}}
\end{equation*}
is finite.
\end{itemize}
\item If $(h)$ holds, we may introduce the function
\begin{equation*}
	\mathbf{H}(r) = \int_r^{+\infty}  \frac{\D t}{h(t)^{N-1}},
\end{equation*}
and it is easy to check that $\mathbf{H}$ is a positive harmonic function on $\mathcal{M}$ (possibly singular at $r=d_g(\sigma,o)=0$, i.e. at the pole $o$) which satisfies
\begin{equation*}
	\lim_{d_g(\sigma,o) \to +\infty} \mathbf{H}(\sigma) = 0.
\end{equation*}
\end{itemize}

We can now state our main existence result for problem \eqref{eq:P}.

\begin{theorem}
  \label{th:main}
  Let $\mathcal{M}$ be an $N$-dimensional model manifold endowed with the metric
  \begin{equation*}
    g = \D r^2 + h(r) \, \D \theta^2,
  \end{equation*}
  where $h$ is a smooth positive function on $[0,+\infty)$ satisfying $(h)$.
  Suppose that
  \begin{itemize}
  \item[$(V_0)$] $V(\sigma) \geq 0$ for every
    $\sigma \in \mathcal{M};$
  \item[$(V_2)$] there exists a constant $V_\infty >0$ such that
    $V(\sigma) \leq V_\infty$ for every $\sigma \in B;$
  \item[$(f_1$)]
    $\limsup_{t \to 0+} \frac{f(t)}{t^{2^*-1}} < +\infty$, where
    $2^* = 2N/(N-2);$
  \item[$(f_2)$] there exists $2<p<2^*$ such that
    \[ \limsup_{t \to +\infty} \frac{f(t)}{t^{p-1}} = 0; \]
  \item[$(f_3)$]  there exists
    $\mu >2$ such that $\mu F(t) \leq t f(t)$ for all $t >0$, where
    $F(t) := \displaystyle\int_0^t f(s) \, ds$.
  \end{itemize}
  Suppose moreover that either
  \begin{itemize}
  \item[$(V_1)$] $\inf_{\mathcal{M}} V >0$
  \end{itemize}
  or
  \begin{itemize}
  \item[$(Vol)$] The function \[ r \mapsto \frac{\displaystyle\left( \int_r^{2r} h(t)^{N-1}\, \D t \right)^{1/N}}{r} \] is bounded from above on $[0,+\infty)$.
  \end{itemize}
  Under these assumptions, there exist $\Lambda >0$ and $R>1$ such that, if
  \begin{equation*}
        \mathbf{H}(r) = \int_r^{+\infty} \frac{\D t}{h(t)^{N-1}},
  \end{equation*}
  and
  \begin{equation*}
		\mathbf{H}(R)^{\frac{4}{N-2}} \inf_{d_g(\sigma,o) \geq R} \frac{V(\sigma)}{\mathbf{H}(d_g(\sigma,o))^{\frac{4}{N-2}}} \geq \Lambda,
              \end{equation*}
    then \eqref{eq:P} possesses at least a nontrivial positive solution.
\end{theorem}
Before proceeding to the proof of Theorem \ref{th:main}, we observe that $V$ may be unbounded at infinity (though locally bounded by condition (V$2$)).
\begin{remark}
Since $p<2^*$, assumptions ($f_1$) and ($f_2$) imply the existence of a constant $c_0>0$ such that
\begin{equation}
	\left\vert s f(s) \right\vert \leq c_0 \vert s \vert^{2^*}, \quad \left\vert s f(s) \right\vert \leq c_0 \vert s \vert^{p}
\end{equation}
for every $s \in \mathbb{R}$.
\end{remark}

\section{The auxiliary problem}

Since we are looking for \emph{positive} solutions to \eqref{eq:P}, we will suppose that $f(t) =0$ for every $t <0$. The main idea behind our approach is based on a suitable modification of the nonlinearity $f$, in such a way that the Palais-Smale condition can be recovered.  To complete the proof, we need to check that the solution of the modified equation is actually a solution of equation \eqref{eq:P}. This technique goes back to the paper \cite{zbMATH00852856}.

\medskip

We introduce the Sobolev space\footnote{We refer the reader to \cite{zbMATH01447265} for a discussion of Sobolev spaces on Riemannian manifolds.} $X$ defined as
\begin{equation*}
	X = \left\lbrace u \in D^{1,2}(\mathcal{M}) : \int_{\mathcal{M}} V \vert u \vert^2 \, \D v_g < +\infty \right\rbrace.
\end{equation*}
\begin{remark}
	Assumption ($V_1$) guarantees that $X$ is continuously embedded into $L^2(\mathcal{M})$. In this case, $X$ may be considered as a subspace of $H_0^1(\mathcal{M})$.
\end{remark}
We introduce the functional $I \colon X \to \mathbb{R}$ as
\begin{equation*}
	I(u) := \frac{1}{2} \Vert u \Vert^2 - \int_{\mathcal{M}} F(u(\sigma))\, \D v_g.
\end{equation*}
The functional $I$ is of class $C^1$ on $\mathcal{M}$ as a standard consequence of assumptions ($f_1$)--($f_3$). Moreover, critical points of $I$ on $\mathcal{M}$ correspond to weak solutions of problem \eqref{eq:P}. Let us set
\begin{equation*}
	k := \frac{2\mu}{\mu-2} > 2,
\end{equation*}
and consider a number $R>1$. We define $\tilde{f} \colon \mathcal{M} \times \mathbb{R} \to \mathbb{R}$ by
\begin{equation*}
	\tilde{f}(\sigma,t) :=
	\begin{cases}
		f(t) &\hbox{if $k f(t) \leq V(\sigma) t$} \\
		\frac{V(\sigma)}{k} t &\hbox{if $kf(t) > V(\sigma) t$},
	\end{cases}
\end{equation*}
and $g \colon \mathcal{M} \times \mathbb{R} \to \mathbb{R}$ by
\begin{equation*}
	g(\sigma,t) := \begin{cases}
				f(t) &\hbox{if $d_g(\sigma,o) \leq R$} \\
				\tilde{f}(\sigma,t) &\hbox{if $d_g(\sigma,o) > R$}.
				\end{cases}
\end{equation*}

We collect the main estimates for the auxiliary functions $\tilde{f}$ and $g$. The proof is standard and therefore omitted.
\begin{lemma} \label{lem:1}
	The following relations hold:
	\begin{enumerate}
		\item[$(i_1)$] $\tilde{f}(\sigma,t) \leq f(t)$ for every $\sigma \in \mathcal{M}$ and $t \in \mathbb{R};$
		\item[$(i_2)$] $g(\sigma,t) \leq \frac{V(\sigma)}{k}t$ for every $\sigma \in \mathcal{M}$ and $t \in \mathbb{R}$ such that $d_g(\sigma,o) \geq R;$
		\item[$(i_3)$] $G(\sigma,t) = F(t)$ for every $\sigma \in \mathcal{M}$ and $t \in \mathbb{R}$ such that $d_g(\sigma,o) \leq R;$
		\item[$(i_4)$] $G(\sigma,t) \leq \frac{V(\sigma)}{2k} t^2$ for every $\sigma \in \mathcal{M}$ and $t \in \mathbb{R}$ such that $d_g(\sigma,o) > R$.
	\end{enumerate}
	Here $G(\sigma,t) := \displaystyle\int_0^t g(\sigma,s)\, ds$.
\end{lemma}

We can now introduce a functional $J \colon X \to \mathbb{R}$ such that
\begin{equation*}
	J(u) := \frac{1}{2} \Vert u \Vert^2 - \int_{\mathcal{M}} G(\sigma,u(\sigma))\, \D v_g.
\end{equation*}
It is immediate to check that $J \in C^1(\mathcal{M})$ and that the G\^{a}teaux derivative of $J$ is given by
\begin{equation*}
	J'(u)[v] = \int_{\mathcal{M}} \left( \nabla_g u \cdot \nabla_g v + V(\sigma) uv \right)\, \D v_g - \int_{\mathcal{M}} g(\sigma,u)v\, \D v_g.
\end{equation*}
Therefore, critical points of $J$ correspond to weak solutions of the equation
\begin{equation}
	-\Delta_g u + V(\sigma) u = g(\sigma,u) \quad\hbox{in $\mathcal{M}$}. \label{eq:AP}
\end{equation}

Let $I_0 \colon H_0^1 (B) \to \mathbb{R}$ be the functional defined by
\begin{equation*}
	I_0(u) := \frac{1}{2} \int_B \left( \vert \nabla_g u \vert^2 + V_\infty \vert u \vert^2 \right)\, \D v_g - \int_B F(u(\sigma))\, \D v_g.
\end{equation*}
We define the \emph{mountain pass level} of $I_0$ as
\begin{equation*}
	d := \inf_{\gamma \in \Gamma} \max_{0 \leq t \leq 1} I_0(\gamma(t)),
\end{equation*}
where
\begin{equation*}
	\Gamma := \left\lbrace \gamma \in C([0,1],H_0^1(B)) \mid \gamma(0)=0, \; \gamma(1) =e \right\rbrace,
\end{equation*}
and $e \in H_0^1(B)$ is such that $I_0(e)<0$. By (3) of Lemma \ref{lem:1} and ($V_2$) we deduce that
\begin{equation*}
	J(u) \leq I_0(u) \quad\hbox{for every $u \in H_0^1(B)$}.
\end{equation*}
It follows immediately that
\begin{gather*}
	c = \inf_{\gamma \in \Gamma} \max_{0 \leq t \leq 1} J(\gamma(t)) \leq d
\end{gather*}

The functional $J$ gains some topological strength from the modified nonlinearity $g$.
\begin{proposition}
	Suppose that either assumption $(V_1)$ or assumption $(Vol)$ holds. Then the functional $J$ satisfies the Palais-Smale condition on $X$.
\end{proposition}
\begin{proof}
	Let $\lbrace u_n \rbrace_n$ be a Palais-Smale sequence for $J$ in $X$, i.e. the sequence $\lbrace J(u_n) \rbrace_n$ is bounded and $J'(u_n) \to 0$ strongly in $X^*$. We compute first
	\begin{equation}
		J(u_n)-\frac{1}{\mu} J'(u_n)[u_n] \geq \frac{\mu-2}{4\mu} \Vert u_n \Vert^2 = \frac{1}{2k} \Vert u_n \Vert^2.
\label{eq:ps-bounded}
	\end{equation}
	The left-hand side of \eqref{eq:ps-bounded} is eventually bounded by $M+\Vert u_n \Vert$ for some constant $M$, and we conclude that
	\begin{equation*}
		\Vert u_n \Vert^2 \leq 2k \left( M + \Vert u_n \Vert \right)
	\end{equation*}
	for $n \gg 1$. Thus the sequence $\lbrace u_n \rbrace_n$ is bounded in $X$. We may assume  that (up to a subsequence) $u_n$ converges weakly to some $u$ in $X$. Fix $\varepsilon >0$, and choose a number $r>R$ such that
	\begin{equation}
          \left( \int_{A(r,2r)} \vert u \vert^{2^*} \, \D v_g \right)^{1/2^*} < \varepsilon,
          \label{eq:3x}
	\end{equation}
	where we have set for $s\geq 0$, $t \geq 0$,
	\begin{equation*}
		A(s,t) := \left\lbrace \sigma \in \mathcal{M} \mid s \leq d_g(\sigma,o) \leq t \right\rbrace
	\end{equation*}
	Recalling the discussion at the and of the proof of \cite[Proposition 4.1]{Shubin_2001}, there exists a smooth cut-off function $\eta=\eta_r$ such that $\operatorname{supp} \eta \subset \mathcal{M} \setminus B_r(o)$, $\eta =1$ on $\mathcal{M} \setminus B_{2r}(o)$, $0 \leq \eta \leq 1$ and $\vert \nabla_g \eta \vert \leq 2/r$ on $\mathcal{M}$. The boundedness of the sequence $\lbrace \eta u_n\rbrace_n$ yields
	\begin{equation*}
		\int_{\mathcal{M}} \left( \nabla_g u_n \cdot \nabla_g (\eta u_n) +  \eta V(\sigma) \vert u_n \vert^2 \right)\, \D v_g = \int_{\mathcal{M}} \eta g(\sigma,u_n)u_n\, \D v_g + o(1).
	\end{equation*}
	Combining with (2) of Lemma \ref{lem:1} we see that
	\begin{multline*}
		\int_{d_g(\sigma,o) \geq r} \eta \left( \vert \nabla_g u_n \vert^2 + V(\sigma) \vert u_n \vert^2 \right) \D v_g \\
		\leq \frac{1}{k} \int_{d_g(\sigma,o) \geq r}  \eta V(\sigma) \vert u_n \vert^2 \, \D v_g - \int_{d_g(\sigma,o) \geq r} u_n \nabla_g u_n \cdot \nabla_g \eta \, \D v_g + o(1).
	\end{multline*}
	As a consequence,
	\begin{multline*}
		\left( 1-\frac{1}{k} \right) \int_{d_g(\sigma,o) \geq r} \eta \left( \vert \nabla_g u_n \vert^2 + V(\sigma) \vert u_n \vert^2 \right) \D v_g \\
		\leq \frac{2}{r} \int_{A(r,2r)} \vert u_n \vert \, \vert \nabla_g u_n \vert \, \D v_g + o(1).
	\end{multline*}
	On the bounded set $A(r,2r)=B_{2r}(o) \setminus B_r(o)$ the Sobolev embedding theorem ensures that $u_n \to u$ in the sense of $L^2$; the H\"{o}lder inequality and the boundedness of $\lbrace u_n \rbrace_n$ yield
	\begin{equation}
\label{eq:3}
		\limsup_{n \to +\infty} \int_{A(r,2r)} \vert u_n \vert \, \vert \nabla_g u_n \vert \, \D v_g \leq C \left( \int_{A(r,2r)} \vert u \vert^2 \, \D v_g \right)^{1/2},
	\end{equation}
	where $C>0$ is a suitable constant. If assumption ($V_1$) holds, then $u \in L^2(\mathcal{M})$ and therefore
\begin{equation*}
	\lim_{r \to +\infty} \int_{A(r,2r)} \vert u\vert^2 \, \D v_g =0.
\end{equation*}
On the other hand, if (Vol) holds, then
	\begin{equation*}
		\left( \int_{A(r,2r)} \vert u \vert^2 \, \D v_g \right)^{1/2} \leq \left( \int_{A(r,2r)} \vert u \vert^{2^*} \, \D v_g \right)^{1/2^*} \operatorname{Vol}(A(r,2r))^{1/N}.
	\end{equation*}
	In order to estimate the last volume, we recall that
	\begin{equation*}
		\operatorname{Vol} (B_\varrho (o)) = \omega_N \int_0^\varrho \vert h(s) \vert^{N-1} \, ds,
	\end{equation*}
	where $\omega_N$ is the $(N-1)$-volume of the sphere $\mathbb{S}^{N-1}$. We may finally write
	\begin{multline*}
		\limsup_{n \to +\infty} \int_{d_g(\sigma,o) \geq r} \vert u_n \vert \, \vert \nabla_g u_n \vert \, \D v_g \\
		\leq 2 C\omega_N^{1/N} \Vert u \Vert \left( \int_{A(r,2r)} \vert u \vert^{2^*} \, \D v_g \right)^{1/2^*} \frac{\left( \int_r^{2r} \vert h(s) \vert^{N-1} \, ds
		\right)^{1/N}}{r} \\
\leq 2 C\omega_N^{1/N} \Vert u \Vert \left( \int_{A(r,2r)} \vert u \vert^{2^*} \, \D v_g \right)^{1/2^*}
	\end{multline*}
In any case we get from \eqref{eq:3x} that
\begin{equation*}
	\limsup_{n \to +\infty} \int_{\mathcal{M} \setminus B_r(o)} \left( \vert\nabla_g u_n \vert^2 + V(\sigma) \vert u_n \vert^2 \right)\, \D v_g < C \varepsilon,
\end{equation*}
which in turn yields the convergence of $\left\lbrace u_n \right\rbrace_n$ to $u$.
\end{proof}
\begin{remark} \label{rem:6}
Before proceeding further, we would like to discuss the role of the assumptions (V$_1$) and (Vol). The analysis of Palais-Smale sequences $\lbrace u_n \rbrace_n$ for the modified function $J$ shows that the compactness condition
\begin{equation*}
    \limsup_{r \to +\infty} \int_{d_g(x,o) \geq r} \left( \vert \nabla_g u_n \vert^2 + V(\sigma) \vert u_n \vert^2 \right) \D v_g =0
\end{equation*}
depends on two competing ingredients: the decay of the \emph{gradient} of the cut-off function $\eta$ with respect to $r$, and the growth of the function $r \mapsto \operatorname{Vol} A(r,2r)$.

In the familiar setting $\mathcal{M}=\mathbb{R}^N$, $\operatorname{Vol} A(r,2r) \lesssim r^N$, which exactly balances the decay $\vert \nabla \eta \vert \lesssim 1/r$ in integration. In our setting, however, the volume of the annulus $A(r,2r)$ is determined by the growth of the function $h$, while the decay of $\eta$ remains the same as in the Euclidean setting. From a technical viewpoint, assumption (V$_1$) states that every element of $X$ lies in $L^2(\mathcal{M})$, so that we can directly bound the integral $$\int_{d_g(x,o) \geq r} \vert u_n \vert \, \vert \nabla_g u_n \vert \, \D v_g$$ by the Cauchy-Schwarz inequality and the size of $\operatorname{Vol} A(r,2r)$ becomes irrelevant. Of course the potential $V$ is no longer allowed to decay to zero at infinity.

On the other hand, assumption (Vol) allows $V$ to decay at infinity, but clearly puts a rather strong restriction on the geometry of the manifold. The survey \cite{carron2020euclidean} contains several sufficient conditions for the Riemannian volume of the ball to grow like in the Euclidean case. In our setting we draw the reader's attention to the following result, see \cite[Theorem G]{carron2020euclidean}; $\rho(\sigma)$ is defined to be the lowest eigenvalue of the Ricci tensor at $\sigma$.
\end{remark}
\begin{theorem}
	Consider $\mathbb{R}^N$ endowed with the metric
	\begin{equation*}
		\D r^2 + h(r)^2 \, \D \theta^2,
	\end{equation*}
	where $h$ is smooth and satisfies $h(0)=0$, $h'(0)=1$. If for some $\lambda \geq (N-1)/4$ the Schrödinger operator $-\Delta+\lambda \rho$ is non-negative, then
	\begin{equation*}
		\operatorname{Vol} B(0,R) \leq c(N,\lambda) R^N.
	\end{equation*}	
\end{theorem}

Once the Palais-Smale condition holds for the functional $J$, the Mountain Pass Theorem \cite{Ambrosetti_1973} guarantees the existence of a critical point
$u \in X$ of the functional $J$ such that $J(u)=c>0$.
\begin{remark}
  If $u$ is a nontrivial critical point of $J$, it follows from the
  relation $c \leq d$ and the estimate
  \begin{equation*}
    \Vert u \Vert^2 \leq 2k \left( J(u)+J'(u)[u] \right)
  \end{equation*}
  coming from \eqref{eq:ps-bounded} that $\Vert u \Vert \leq \sqrt{2kd}$, and this upper bound  does not depend on the parameter $R>1$.
\end{remark}

The next regularity result is based on an argument reminiscent of the De Giorgi-Nash-Moser iteration technique.
\begin{proposition}
  \label{prop:estimate} Suppose $a \in L^q (\mathcal{M})$ for some $q>N/2$ and that
  $v \in X$ is a weak solution of the equation
  \begin{equation*}
    -\Delta_g v + b(\sigma) v = A(\sigma,v) \quad\hbox{in $\mathcal{M}$},
  \end{equation*}
  where $A \colon \mathcal{M} \times \mathbb{R} \to \mathbb{R}$ is continuous and satisfies
  \begin{equation*}
    \left\vert A(\sigma,t) \right\vert \leq a(x) t \quad\hbox{for all $t >0$}
  \end{equation*}
  and $b\geq 0$ is a $($measurable$)$  function on $\mathcal{M}$. Then there exists a
  constant $M>0$, depending on $q$ and on $\Vert a \Vert_{L^q}$ only, such
  that
  \begin{equation*}
    \left\Vert v \right\Vert_{L^\infty} \leq M \left\Vert v \right\Vert_{L^{2^*}}.
    \end{equation*}
  \end{proposition}
  \begin{proof}
     Since the proof is  standard, we only give an outline and we refer to \cite{Alves_2012} for further details. Fix $\beta >1$ and $m \in \mathbb{N}$. Set
     \begin{equation*}
         w_{m}:= \begin{cases}
             v|v|^{\beta-1} & \hbox{on $A_m$} \\
             mv & \hbox{elsewhere}
         \end{cases}
         \quad \mbox{where} \quad  A_m:= \left\{  \sigma \in \mathcal{M}  \mid |v|^{\beta-1} \leq m \right\}.
     \end{equation*}
     Utilizing the definition of weak solution, the Sobolev inequality and the assumption on the potential it is possible to derive the estimate
        \begin{equation*}
            \left[ \int_{A_m} |w_m|^{2^*} \, \D v_g\right]^{\frac{N-2}{N}} \leq S \beta^2 \int_{\mathcal{M}} a(\sigma) w_m^2\, \D v_g
        \end{equation*}
     where $S>0$ denotes the optimal Sobolev constant. From this, applying the H\"older inequality with $1/q_1+1/q=1$ and letting $m \to \infty$, we get
    \begin{equation} \label{eq:4}
        \Vert v \Vert_{2^* \beta} \leq \beta ^{\frac{1}{\beta}} \left( S \Vert a \Vert_q \right)^{\frac{1}{2 \beta}} \Vert v \Vert_{2 \beta q_1}.
    \end{equation}
    At this point, since $N/(N-2)>q_1$, we set $\gamma:= N/q_1(N-2)>1$. If $\beta=\gamma$, \eqref{eq:4} becomes

\begin{equation} \label{eq:5}
        \Vert v \Vert_{2^* \gamma} \leq \beta ^{\frac{1}{\gamma}} \left( S \Vert a \Vert_q \right)^{\frac{1}{2 \gamma}} \Vert v \Vert_{2^*},
    \end{equation}
    while if $\beta=\gamma^2$, taking into account $2 \beta q_1=2^*\gamma$, we obtain
    \begin{equation} \label{eq:6}
        \Vert v \Vert_{2^* \gamma^2} \leq \gamma^{\frac{2}{\gamma}} \left( S \Vert a \Vert_q \right)^{\frac{1}{2 \gamma^2}} \Vert v \Vert_{2^* \gamma}.
    \end{equation}
    Coupling \eqref{eq:5} and \eqref{eq:6}, we get
    \begin{equation*}
        \Vert v \Vert_{2^* \gamma^2} \leq \gamma ^{\frac{1}{\gamma} + \frac{2}{\gamma^2}} \left( S \Vert a \Vert_q \right)^{\frac{1}{2} \left(\frac{1}{\gamma}+\frac{1}{\gamma^2}\right)} \Vert v \Vert_{2^*}.
    \end{equation*}
    Iterating this procedure $j$ times with $\beta=\gamma^j$, we have
    \begin{equation*}
        \Vert v \Vert_{2^* \gamma^j} \leq \gamma^{\frac{1}{\gamma} + \ldots +  \frac{j}{\gamma^j}} \left( S \Vert a \Vert_q \right)^{\frac{1}{2} \left(\frac{1}{\gamma}+ \ldots +\frac{1}{\gamma^j}\right)} \Vert v \Vert_{2^*}.
    \end{equation*}
    Now, recalling that
    \begin{equation*}
        \sum_{j=1}^\infty \frac{j}{\gamma^j}= \frac{\gamma}{(\gamma-1)^2} \quad \mbox{and} \quad \sum_{j=1}^\infty \frac{1}{\gamma^j}=\frac{1}{\gamma-1}
        \end{equation*}
        and that
        \begin{equation*}
            \lim_{j \to \infty} \Vert v \Vert_{2^* \gamma^j}=\Vert v \Vert_{\infty},
        \end{equation*}
        the proposition is proved selecting
        \begin{equation*}
            M=\gamma^{\frac{\gamma}{(\gamma-1)^2}} \left( S \Vert a \Vert_q \right)^{\frac{1}{2}\frac{1}{\gamma-1}}.
        \end{equation*}
The proof is now complete.
  \end{proof}
\begin{corollary}
	\label{cor:estimate} Any positive ground state of \eqref{eq:AP} satisfies the estimate
	\begin{equation}
		\left\Vert u \right\Vert_{L^\infty} \leq M \sqrt{2Skd},
	\end{equation}
where $S$ is the best constant for the Sobolev embedding $D^{1,2}(\mathcal{M}) \subset L^{2^*}(\mathcal{M})$.
\end{corollary}
\begin{proof}
	Indeed, we consider the functions
	\begin{equation*}
		A(\sigma,t) := \begin{cases}
			f(t) &\hbox{if $d_g(\sigma,o) <R$ or $f(t) \leq \frac{V(\sigma)}{k} t$}, \\
			0 &\hbox{if $d_g(\sigma,o) \geq R$ or $f(t) > \frac{V(\sigma)}{k} t$}
		\end{cases}
	\end{equation*}
and
\begin{equation*}
	b(\sigma) := \begin{cases}
		V(\sigma) &\hbox{if $d_g(\sigma,o) <R$ or $f(t) \leq \frac{V(\sigma)}{k} t$}, \\
		\left( 1 - \frac{1}{k} \right) V(\sigma) &\hbox{if $d_g(\sigma,o) \geq R$ or $f(t) > \frac{V(\sigma)}{k} t$}.
	\end{cases}
\end{equation*}
Any positive solution $u$ to \eqref{eq:AP} satisfies the equation
\begin{equation*}
	-\Delta_g u + b(\sigma)u = A(\sigma,u) \quad \hbox{in $\mathcal{M}$}.
\end{equation*}	
Our assumptions on $f$ yield that $\vert A(\sigma,t) \vert \leq f(t) \leq c_0  \vert t \vert^{p-1}$, hence
\begin{equation*}
	\left\vert A(\sigma,t) \right\vert \leq a(\sigma) \vert t \vert \quad \hbox{with $a(\sigma) = c_0 \vert u(\sigma) \vert^{p-2}$.}
\end{equation*}
 For $q=2^*/(p-2)$ it is immediate to check that $a \in L^q(\mathcal{M})$ and
\begin{equation*}
	\left\Vert a \right\Vert_{L^q} \leq c_0 \left( 2Skd\right)^{\frac{p-2}{2}}.
\end{equation*}
The conclusion follows from Proposition \ref{prop:estimate}.
\end{proof}

\section{Proof of Theorem \ref{th:main}}

\begin{proposition}
	If  $u$ is a positive ground state solution to \eqref{eq:AP}, then
	\begin{equation}
		u(\sigma) \leq \frac{\mathbf{H}(d_g(\sigma,o))}{\mathbf{H}(R)} \Vert u \Vert_{L^\infty} \leq \frac{\mathbf{H}(d_g(\sigma,o))}{\mathbf{H}(R)} M \sqrt{2Skd}
	\end{equation}
	whenever $d_g(\sigma,o) \geq R$.
\end{proposition}
\begin{proof}
	Indeed, we know that the function $\mathbf{H}$ is harmonic on $\mathcal{M}$, and so is the function
	\begin{equation*}
		v  \colon \sigma \mapsto M \sqrt{2Skd} \; \frac{\mathbf{H}(d_g(\sigma,o))}{\mathbf{H}(R)} .
	\end{equation*}
Since $u \leq v$ whenever $d_g(\sigma,o) \geq R$ by Corollary \ref{cor:estimate}, we are allowed to define $\omega \in D^{1,2}(\mathcal{M})$ as
\begin{equation*}
	\omega(\sigma) := \begin{cases}
		\left( u-v \right)^{+} &\hbox{if $d_g(\sigma,o) \geq R$} \\
		0 &\hbox{otherwise}.
	\end{cases}
\end{equation*}
We then see that
\begin{multline*}
	\int_{\mathcal{M}} \left\vert \nabla_g \omega \right\vert^2 \, \D v_g \\
	= \int_{\mathcal{M}} \nabla_g (u-v) \cdot \nabla_g \omega \, \D v_g = \int_{d_g(\sigma,o) \geq R} \left( g(\sigma,u)\omega - V(\sigma)u\omega \right) \D v_g \\
	\leq \left( \frac{1}{k}-1 \right) \int_{\mathcal{M}} V(\sigma) u\omega \, \D v_g \leq 0.
\end{multline*}
It follows that $\omega =0$ on $\mathcal{M}$, and $u \leq v$ whenever $d_g(\sigma,o) \geq R$.
\end{proof}

\begin{proof}[Proof of Theorem \ref{th:main}]
	Let $u \in X$ be a positive ground state solution of \eqref{eq:AP}. For every $\sigma \in \mathcal{M}$ such that $d_g(\sigma,o) \geq R$, we have for any
	\begin{equation*}
		\Lambda >  k c_0 M^{\frac{4}{N-2}} \left( 2Skd \right)^{\frac{2}{N-2}},
	\end{equation*}
	the estimate
	\begin{eqnarray*}
		\frac{f(u)}{u} &\leq& c_0 \left\vert u \right\vert^{\frac{4}{N-2}} \leq c_0 M^{\frac{4}{N-2}} \left( 2Skd \right)^{\frac{2}{N-2}} \left(\frac{\mathbf{H}(d_g(\sigma,o))}{\mathbf{H}(R)} \right)^{\frac{4}{N-2}} \\
		&\leq& \frac{V(\sigma)}{k}.
	\end{eqnarray*}
 It now follows that $u$ solves \eqref{eq:P}, and the proof is complete.
\end{proof}

\bibliographystyle{amsplain}
\bibliography{biblio}

\end{document}